\documentclass[12pt]{amsart}
\usepackage[utf8]{inputenc}
\usepackage{amssymb, amsmath}
\usepackage{fourier}
\usepackage{xcolor}
\usepackage{enumitem}
\usepackage{fancyvrb}

\usepackage{geometry}
\theoremstyle{plain}

\newtheorem{theorem}{Theorem}

\newtheorem{definition}[theorem]{Definition}

\newtheorem{proposition}[theorem]{Proposition}

\theoremstyle{remark}

\newtheorem{example}[theorem]{Example}

\newcommand{\bit}{\begin{itemize}}
\newcommand{\eit}{\end{itemize}}
\newcommand{\ben}{\begin{enumerate}}
\newcommand{\een}{\end{enumerate}}
\newcommand{\be}{\begin{equation}}
\newcommand{\ee}{\end{equation}}
\newcommand{\ba}{\begin{array}}
\newcommand{\ea}{\end{array}}

\newcommand{\bec}{\begin{equation*}}
\newcommand{\eec}{\end{equation*}}

\newcommand{\dd}{\mathrm{d}}
\newcommand{\dt}{\mathrm{d}t}
\newcommand{\ds}{\mathrm{d}s}

\DeclareMathOperator*{\argmin}{arg\,min}

\newcommand{\abs}[1]{\left|#1\right|}
\newcommand{\norm}[1]{\left|\left|#1\right|\right|}

\newcommand\cA{\mathcal A}
\newcommand\cB{\mathcal B}
\newcommand\cH{\mathcal H}

\newcommand\R{\mathbb R}

\newcommand\bh{\cB(\cH)}
\newcommand{\fel}{\frac{1}{2}}
\newcommand{\D}{\mathbf{D}}

\newcommand\tr{\operatorname{Tr}}
\newcommand{\ler}[1]{\left( #1 \right)}
\newcommand\bA{\mathbf A}
\newcommand\bw{\mathbf w}

\newcommand{\lers}[1]{\left\{ #1 \right\}}

\begin{document}

\title[A divergence center interpretation of general Kubo-Ando means]{A divergence center interpretation of general symmetric Kubo-Ando means, and related weighted multivariate operator means}

\author{J\'ozsef Pitrik}
\address[J\'ozsef Pitrik]{MTA-BME Lend\"ulet (Momentum) Quantum Information Theory Research Group, and Department of Analysis, Institute of Mathematics\\
Budapest University of Technology and Economics\\
H-1521 Budapest, Hungary}
\email{pitrik@math.bme.hu}
\urladdr{http://www.math.bme.hu/\~{}pitrik}

\author{D\'aniel Virosztek}
\address[D\'aniel Virosztek]{Institute of Science and Technology Austria\\
Am Campus 1, 3400 Klosterneuburg, Austria}
\email{daniel.virosztek@ist.ac.at}
\urladdr{http://pub.ist.ac.at/\~{}dviroszt}

\thanks{J. Pitrik was supported by the Hungarian Academy of Sciences Lend\"ulet-Momentum Grant for Quantum Information Theory, No. 96 141, and by the Hungarian National Research, Development and Innovation Office (NKFIH) via grants no. K119442, no. K124152, and no. KH129601. D. Virosztek was supported by the ISTFELLOW program of the Institute of Science and Technology Austria (project code IC1027FELL01), by the European Union’s Horizon 2020 research and innovation program under the Marie Sklodowska-Curie Grant Agreement No. 846294, and partially supported by the Hungarian National Research, Development and Innovation Office (NKFIH) via grants no. K124152, and no. KH129601.
}



\begin{abstract}
It is well known that special Kubo-Ando operator means admit divergence center interpretations, moreover, they are also mean squared error estimators for certain metrics on positive definite operators. In this paper we give a divergence center interpretation for every symmetric Kubo-Ando mean. This characterization of the symmetric means naturally leads to a definition of weighted and multivariate versions of a large class of symmetric Kubo-Ando means. We study elementary properties of these weighted multivariate means, and note in particular that in the special case of the geometric mean we recover the weighted $\mathcal{A} \# \mathcal{H}$-mean introduced by Kim, Lawson, and Lim \cite{kim-lawson-lim}.
\end{abstract}
\maketitle

\section{Introduction} \label{sec:intro}
\subsection{Motivation, goals} \label{susec:motiv-goals}
It is well known that special Kubo-Ando operator means, namely, the arithmetic, the geometric, and the harmonic mean admit divergence center interpretations, moreover, as the divergence may be a squared distance, they are also mean squared error estimators.
The arithmetic mean
$A \nabla B=(A+B)/2$ is clearly the mean squared error estimator for the Euclidean metric on positive operators:
$$
A \nabla B= \argmin_{X>0} \fel \ler{\tr (A-X)^2+\tr (B-X)^2}.
$$
The geometric mean $A\#B=A^{1/2}(A^{-1/2}BA^{-1/2})^{1/2}A^{1/2}$ is the mean squared error estimator for both the Riemannian trace metric $d_R(X,Y)=\norm{\log{\ler{X^{-\fel}Y X^{-\fel}}}}_2,$ and the $S$-divergence $d_S(X,Y)=\sqrt{\tr \log \ler{\frac{X+Y}{2}}-\fel \tr \log X -\fel \tr \log Y}$ that is,
$$
A \# B= \argmin_{X>0} \fel \ler{d_R^2(A,X)+d_R^2(B,X)}= \argmin_{X>0} \fel \ler{d_S^2(A,X)+d_S^2(B,X)},
$$
see \cite{bhat-holb} and \cite{sra-pams} for the details, respectively.
The harmonic mean $A ! B= 2 \ler{A^{-1}+B^{-1}}^{-1}$ is the mean squared error estimator for a Riemannian metric on positive definite operators described in the Appendix in details.
Furthermore, all the above mentioned means have a standard weighted version. The curve consisting of weighted arithmetic means
$$
A \nabla_\alpha B=(1-\alpha) A +\alpha B \qquad \ler{\alpha \in [0,1]}
$$
is the geodesic in the Euclidean metric, the weighted geometric means
$$
A \#_\alpha B = A^{1/2}(A^{-1/2}BA^{-1/2})^{\alpha}A^{1/2} \qquad \ler{\alpha \in [0,1]}
$$
form the geodesic between $A$ and $B$ in the Riemannian trace metric, and the flow of weighted harmonic means
$$
A !_\alpha B= \ler{(1-\alpha)A^{-1}+\alpha B^{-1}}^{-1}
$$
is the geodesic with respect to the Riemannian metric given by
$$
\ds=\|A^{-1}\dd A A^{-1}\|_2
$$
see the Appendix for the details.

\par
In this paper we give a divergence center interpretation for every symmetric Kubo-Ando mean. This characterization of the symmetric means immediately leads to a natural definition of weighted and multivariate versions of a large class of symmetric Kubo-Ando means. We study elementary properties of these weighted multivariate means, and note in particular that in the special case of the geometric mean we recover the weighted $\mathcal{A} \# \mathcal{H}$-mean introduced by Kim, Lawson, and Lim \cite{kim-lawson-lim}.

\subsection{Basic notions, notation} \label{susec:b-not-not}
Throughout this paper, $\cH$ stands for a finite dimensional complex Hilbert space, and $\bh,$ $\bh^{sa},$ and $\cB(\cH)^{++}$ denote the sets of all linear, self-adjoint, and positive definite operators on $\cH,$ respectively. The symbol $I$ stands for the identity of $\bh,$ and we consider the L\"owner order induced by positivity on $\bh^{sa},$ that is, by $A \leq B$ we mean that $B-A$ is positive semidefinite, and $A<B$ means that $B-A$ is positive definite. The spectrum of $X \in \bh$ is denoted by $\mathrm{spec}(X).$ The symbols $\D$ and $\D^2$ denote the first and second Fr\'echet derivatives, respectively.
\section{Symmetric Kubo-Ando means as divergence centers} \label{sec:ka-msee}

Let $\sigma: \, \bh^{++} \times \bh^{++} \rightarrow \bh^{++}$ be a symmetric Kubo-Ando operator mean, and let $f_\sigma: \, (0, \infty) \rightarrow (0,\infty)$ be the operator monotone function representing $\sigma$ in the sense that
\be \label{eq:f-sig-rep}
A \sigma B= A^{\fel}f_\sigma \ler{A^{-\fel} B A^{-\fel}} A^{\fel},
\ee
see \cite{kubo-ando} for an overview of the theory of operator means.
Clearly, $f_\sigma(1)=1,$ and the symmetry of $\sigma$ implies that $f_\sigma(x)=x f_\sigma\ler{\frac{1}{x}}$ for $x>0,$ and hence $f_\sigma'(1)=1/2$. 
We define
$$
g_\sigma: \, (0, \infty) \supseteq \mathrm{ran}\ler{f_\sigma} \rightarrow [0, \infty)
$$
by
\be\label{eq:g}
g_\sigma(x):=\int_1^x\left(1-\frac{1}{f_\sigma^{-1}(t)}\right)\dd t.
\ee
Obviously, $g_\sigma(1)=0$, $g_\sigma'(x)=1-\frac{1}{f_\sigma^{-1}(x)},$ and $g_\sigma'(1)=0$ as $f_\sigma(1)=1$. Since $f_\sigma$ is strictly monotone increasing, so is $g_\sigma'$, and hence $g_\sigma$ is strictly convex on its domain. 
Now we define the quantity
\be \label{eq:fi-sig}
\phi_\sigma (A,B):=\tr g_\sigma\left( A^{-1/2}BA^{-1/2}\right),
\ee
for positive definite operators $A,B\in\cB(\cH)^{++}$ such that the spectrum of $A^{-1/2}BA^{-1/2}$ is contained in $\mathrm{ran}\ler{f_\sigma}.$ We define $\phi_\sigma(A,B):= +\infty$ if $\mathrm{spec} \ler{A^{-1/2}BA^{-1/2}} \nsubseteq \mathrm{ran}\ler{f_\sigma}.$
It will be important in the sequel that by \cite[2.10. Thm.]{carlen} the strict convexity of $g_\sigma$ implies that $X \mapsto \phi_\sigma(A,X)$ is strictly convex (whenever finite) for every $\sigma$ and $A.$
Now we check that $\phi_\sigma$ defined in \eqref{eq:fi-sig} is a genuine divergence in the sense of Amari \cite[Sec. 1.2 \& 1.3]{amari-book}.

\begin{proposition} \label{prop:div-valid}
For any symmetric Kubo-Ando mean $\sigma,$ the map
\be \label{eq:phi-mu-map}
\phi_\sigma: \, \bh^{++} \times \bh^{++} \rightarrow [0,+\infty]; \quad (A,B) \mapsto \phi_\sigma(A,B)
\ee
satisfies the followings.
\ben[label=(\roman*)]
\item \label{pr:1definit}
$\phi_\sigma(A,B)\geq 0$ and $\phi_\sigma(A,B)=0$ if and only if $A=B.$
\item \label{pr:2vanish}
The first derivative of $\phi_\sigma$ in the second variable vanishes at the diagonal, that is,
$\D\ler{\phi_\sigma(A,\cdot)}[A]=0 \in \mathrm{Lin}\ler{\bh^{sa},\R}$ for all $A\in \bh^{++}.$
\item \label{pr:3posdef}
The second derivative of $\Phi_\sigma$ in the second variable is positive at the diagonal, that is,
$\D^2 \ler{\phi_\sigma(A,\cdot)}[A](Y,Y) \geq 0$ for all $Y \in \bh^{sa}.$
\een
\end{proposition}

\begin{proof}
Property \ref{pr:1definit} follows from that $f_\sigma$ is strictly monotone increasing for any symmetric Kubo-Ando mean $\sigma,$ and hence so is $g_\sigma'(x)=1-\frac{1}{f_\sigma^{-1}(x)}.$ Moreover, $g_\sigma'(1)=g_\sigma(1)=0,$ and therefore $g_\sigma(x)=0$ if and only if $x=1.$ Consequently, $\phi_\sigma(A,B)=\tr g_\sigma \ler{A^{-\fel}BA^{-\fel}}=0$ if and only if $A^{-\fel}BA^{-\fel}=I.$
\par
Property \ref{pr:2vanish} can be seen by calculating that the derivative of $\phi_\sigma \ler{A, \cdot}$ at the diagonal is
$$
\D\ler{\phi_\sigma(A,\cdot)}[A](Y)=\tr g_\sigma'\ler{A^{-\fel} A A^{-\fel}} A^{-\fel} Y A^{-\fel}=0
$$
as $g_\sigma'(1)=0.$
\par
To justify Property \ref{pr:3posdef} note that the second derivative at the diagonal can be calculated as
$$
\D^2 \ler{\phi_\sigma(A,\cdot)}[A](Y,Y)
=\frac{\dd^2}{\dt^2}\phi_\sigma\ler{A, A+tY}_{|t=0}
$$
$$
=\frac{\dd^2}{\dt^2}\tr g_\sigma \ler{I+t A^{-\fel} Y A^{-\fel}}_{|t=0}
=g_\sigma''(1) \tr \ler{A^{-\fel}YA^{-\fel}}^2 \geq 0
$$
because $g_\sigma''(1)=\ler{f_\sigma^{-1}(1)}^{-2}\ler{f_\sigma^{-1}}'(1)=2$ is positive.
\par
The standard computation rules for derivatives of functions defined on operators which were used throughout this proof are collected e.g. in \cite[Thm. 3.33.]{HPbook}.
\end{proof}

In what follows, we consider the symmetric loss function $X \mapsto \frac{1}{2}\left(\phi_\sigma (A,X)+\phi_\sigma(B,X)\right)$ corresponding to the pair $A,B \in \bh^{++}$ and to the divergence $\phi_\sigma.$
Note that the domain
$$
\mathrm{dom}_{A,B, \sigma}
=\lers{X \in \bh^{++} \middle|  \fel \left(\phi_\sigma (A,X)+\phi_\sigma(B,X)\right) < \infty}
$$
is a nonempty open convex subset of $\bh^{++}.$ Indeed, $A \sigma B = B \sigma A \in \mathrm{dom}_{A,B, \sigma}$ always holds, because 
$$
\mathrm{spec}\ler{A^{-\fel} \ler{A \sigma B} A^{-\fel}}
=\mathrm{spec}\ler{f_\sigma \ler{A^{-\fel}B A^{-\fel}}} \subset \mathrm{ran}\ler{f_\sigma}
$$
and
$$
\mathrm{spec}\ler{B^{-\fel} \ler{B \sigma A} B^{-\fel}}
=\mathrm{spec}\ler{f_\sigma \ler{B^{-\fel}A B^{-\fel}}} \subset \mathrm{ran}\ler{f_\sigma}.
$$
The convexity of $\mathrm{dom}_{A,B, \sigma}$ follows from the fact that
$$
\mathrm{spec}\ler{A^{-\fel}\ler{(1-\alpha)X+\alpha Y}A^{-\fel}} \subseteq \mathbf{ch}\ler{\mathrm{spec}\ler{A^{-\fel} X A^{-\fel}} \cup \mathrm{spec}\ler{A^{-\fel} Y A^{-\fel}}}
$$
holds for any $X, Y \in \bh^{++}$ and for any $\alpha \in [0,1],$ where $\mathbf{ch}(S)$ denotes the convex hull of the set $S,$ and the same holds when replacing $A$ by $B.$
So if both $X$ and $Y$ are in $\mathrm{dom}_{A,B, \sigma},$ that is, 
$$
\mathrm{spec}\ler{A^{-\fel} X A^{-\fel}} \cup \mathrm{spec}\ler{B^{-\fel} X B^{-\fel}} \cup \mathrm{spec}\ler{A^{-\fel} Y A^{-\fel}} \cup
\mathrm{spec}\ler{B^{-\fel} Y B^{-\fel}} \subset \mathrm{ran}\ler{f_\sigma},
$$
then any convex combination of $X$ and $Y$ is in $\mathrm{dom}_{A,B, \sigma}.$
To see that $\mathrm{dom}_{A,B, \sigma}$ is open, note that $\mathrm{ran}\ler{f_\sigma}$ is an open interval and $\mathrm{spec}\ler{A^{-\fel} X A^{-\fel}}$ is a finite set. Therefore if 
$\mathrm{spec}\ler{A^{-\fel} X A^{-\fel}} \subset \mathrm{ran}\ler{f_\sigma}$ then there is a small neighbourhood of $X$ such that $\mathrm{spec}\ler{A^{-\fel} Y A^{-\fel}} \subset \mathrm{ran}\ler{f_\sigma}$ for any element $Y$ in this neighbourhood (and the same holds when replacing $A$ by $B$). 
\par
We turn to the main result of this section which says that the Kubo-Ando mean $A\sigma B$ is exactly the divergence center (or barycenter) of $A$ and $B$ with respect to the divergence $\phi_\sigma.$
\begin{theorem} \label{thm:bary}
For any $A,B\in\cB(\cH)^{++},$
\be \label{eq:bary}
\argmin_{X\in\cB(\cH)^{++}}\frac{1}{2}\left(\phi_\sigma (A,X)+\phi_\sigma(B,X)\right)=A\sigma B.
\ee
That is, $A\sigma B$ is a unique minimizer of the function $X\mapsto \frac{1}{2}\left(\phi_\sigma (A,X)+\phi_\sigma(B,X)\right)$ on $\cB(\cH)^{++}$.
\end{theorem}
\begin{proof}
By the strict convexity of $X\mapsto \frac{1}{2}\left(\phi_\sigma (A,X)+\phi_\sigma(B,X)\right)$ it is sufficient to show that $A \sigma B$ is a critical point, and therefore a unique minimizer. First we compute the derivative 
\bec
\left.\frac{\dd}{\dd t}\right|_{t=0}\phi_\sigma(A,X+tY)=\left.\frac{\dd}{\dd t}\right|_{t=0}\tr g_\sigma\left(A^{-1/2}XA^{-1/2}+tA^{-1/2}YA^{-1/2}\right)=
\eec
\be\label{eq:derivative}
\tr A^{-1/2}g_\sigma'\left(A^{-1/2}XA^{-1/2}\right)A^{-1/2}Y
\ee
for all $Y\in\cB(\cH)^{sa}$. Since $g_\sigma'(x)=1-(f_\sigma^{-1}(x))^{-1}$, we get
\be\label{eq:derivative2}
\left.\frac{\dd}{\dd t}\right|_{t=0}\frac{1}{2}\phi_\sigma(A,X+tY)=\frac{1}{2}\tr \left(A^{-1/2}\left(I-\left[f_\sigma^{-1}\left(A^{-1/2}XA^{-1/2}\right)\right]^{-1}\right)A^{-1/2}Y\right)
\ee
for all $Y\in\cB(\cH)^{sa}$. Substituting $X=A\sigma B=A^{1/2}f_\sigma(A^{-1/2}BA^{-1/2})A^{1/2}$ into the derivative above, the right hand side of (\ref{eq:derivative2}) becomes
\bec
\frac{1}{2}\tr \left(A^{-1/2}\left(I-\left[f_\sigma^{-1}\left(A^{-1/2}A^{1/2}f_\sigma(A^{-1/2}BA^{-1/2})A^{1/2}A^{-1/2}\right)\right]^{-1}\right)A^{-1/2}Y\right)=
\eec
\be\label{eq:der3}
\frac{1}{2}\tr(A^{-1}-B^{-1})Y.
\ee
Since the operator mean $\sigma$ is symmetric, that is 
$$
A\sigma B=B\sigma A=B^{1/2}f_\sigma\left(B^{-1/2}AB^{-1/2}\right)B^{1/2},
$$
a similar computation for the derivative 
$$
\left.\frac{\dd}{\dd t}\right|_{t=0}\frac{1}{2}\phi_\sigma(B,X+tY)$$ 
at $X=A\sigma B$ gives
\be\label{eq:der4}
\frac{1}{2}\tr(B^{-1}-A^{-1})Y
\ee
for all $Y\in\cB(\cH)^{sa}$. Using (\ref{eq:der3}) and (\ref{eq:der4}) we get for the derivative
\bec
\left.\left.\frac{\dd}{\dd t}\right|_{t=0}\left\lbrace\frac{1}{2}\phi_\sigma(A,X+tY)+\frac{1}{2}\phi_\sigma(B,X+tY)\right\rbrace\right|_{X=A\sigma B}
\eec
\bec
=\frac{1}{2}\tr(A^{-1}-B^{-1})Y+\frac{1}{2}\tr(B^{-1}-A^{-1})Y=0
\eec
for all $Y\in\cB(\cH)^{sa}$. So we obtained that $A \sigma B$ is a critical point and hence a unique minimizer of $X\mapsto \frac{1}{2}\left(\phi_\sigma (A,X)+\phi_\sigma(B,X)\right).$
\end{proof}

\section{Weighted multivariate versions of Kubo-Ando means}

The above characterization of symmetric Kubo-Ando means as barycenters (Theorem \ref{thm:bary}) naturally leads to the idea of defining weighted and multivariate versions of Kubo-Ando means as minimizers of appropriate loss functions derived from the divergence $\phi_\sigma.$
\par
Given a symmetric Kubo-Ando mean $\sigma,$ a finite set of positive definite operators $\bA=\lers{A_1, \dots, A_m} \subset \bh^{++},$ and a discrete probability distribution $\bw=\lers{w_1, \dots, w_m} \subset (0,1]$ with $\sum_{j=1}^m w_j=1$ we define the corresponding loss function $Q_{\sigma, \bA, \bw}: \, \bh^{++} \rightarrow [0, \infty]$ by
\be \label{eq:Q-def}
Q_{\sigma, \bA, \bw}(X):=\sum_{j=1}^m w_j \phi_\sigma \ler{A_j,X}
\ee
where $\phi_\sigma$ is defined by \eqref{eq:fi-sig}.
\par
However, in the weighted multivariate setting, when $\mathrm{ran}\ler{f_\sigma}$ is smaller than the whole positive half-line $(0, \infty),$ then some undesirable phenomena occur which are illustrated by the next example.
\par
Consider the arithmetic mean generated by $f_\nabla(x)=(1+x)/2$ with $\mathrm{ran}\ler{f_\nabla}=\ler{\fel, \infty}.$ Let $A_1, A_2 \in \bh^{++}$ satisfy $ A_1 < \frac{1}{3} A_2.$ In this case, for any $\alpha \in (0,1),$ the loss function $Q_{\nabla, \lers{A_1, A_2}, \lers{1-\alpha, \alpha}} (X)$ is finite only if $X>\fel A_2.$ So the barycenter of $A_1$ and $A_2$ with weights $\lers{1-\alpha,\alpha}$ is separated from $A_1$ for every $\alpha \in (0,1),$ even for values very close to $0.$
\par
To exclude such phenomena, from now on, we assume that the range of $f_\sigma$ is maximal, that is, $\mathrm{ran}\ler{f_\sigma}=(0, \infty),$ and hence
$g_\sigma(x)=\int_1^x\left(1-\frac{1}{f_\sigma^{-1}(t)}\right)\dd t$ is defined on the whole positive half-line $(0, \infty).$ Consequently, $\phi_\sigma$ is always finite, and hence so is $Q_{\sigma, \bA, \bw}$ on the whole positive definite cone $\bh^{++}.$

\begin{definition}\label{def:bary_alpha}
Let $\sigma: \, \bh^{++} \times \bh^{++} \rightarrow \bh^{++}$ be a symmetric Kubo-Ando operator mean such that $f_\sigma: \, (0, \infty) \rightarrow (0,\infty),$ which is the operator monotone function representing $\sigma$ in the sense of \eqref{eq:f-sig-rep}, is surjective. Let $g_\sigma$ be defined as in \eqref{eq:g}, and $\phi_\sigma$ be defined as in \eqref{eq:fi-sig}.
We call the optimizer
\be \label{eq:bary_alpha}
\mathbf{bc}\ler{\sigma,\bA, \bw}:=\argmin_{X\in\cB(\cH)^{++}}Q_{\sigma, \bA, \bw}
\ee
the weighted barycenter of the operators $\lers{A_1, \dots, A_m}$  with weights $\lers{w_1, \dots, w_m}$.
\end{definition}
By Theorem \ref{thm:bary}, this barycenter may be considered as a weighted multivariate version of Kubo-Ando means.
\par
To find the barycenter $\mathbf{bc}\ler{\sigma, \bA, \bw},$ we have to solve the critical point equation
\be \label{eq:crit-point}
\D Q_{\sigma, \bA, \bw}[X](\cdot)=0
\ee
for the strictly convex loss function $Q_{\sigma, \bA, \bw},$ where the symbol
$$
\D Q_{\sigma, \bA, \bw}[X](\cdot) \in \mathrm{Lin}\ler{\bh^{sa}, \R}
$$
stands for the Fr\'echet derivative of $Q_{\sigma, \bA, \bw}$ at the point $X \in \bh^{++}.$
\par
For any $Y \in \bh^{sa}$ we have
$$
\D Q_{\sigma, \bA, \bw}[X](Y)
= \sum_{j=1}^m w_j \D \ler{\phi_\sigma \ler{A_j,\cdot}}[X](Y)
= \sum_{j=1}^m w_j \D \ler{\tr g_\sigma\ler{A_j^{-\fel} \cdot A_j^{-\fel}}}[X](Y)
$$
$$
= \sum_{j=1}^m w_j \tr g_\sigma'\ler{A_j^{-\fel} X A_j^{-\fel}}A_j^{-\fel} Y A_j^{-\fel}
$$
$$
= \tr \ler{\sum_{j=1}^m w_j A_j^{-\fel} g_\sigma'\ler{A_j^{-\fel} X A_j^{-\fel}}A_j^{-\fel}} Y 
$$
that is, the equation to be solved is
\be \label{eq:crit-point-2}
\sum_{j=1}^m w_j A_j^{-\fel} g_\sigma'\ler{A_j^{-\fel} X A_j^{-\fel}}A_j^{-\fel}=0 .
\ee
By the definition of $g_\sigma,$ see \eqref{eq:g}, $g_\sigma'(t)=1-\frac{1}{f_\sigma^{-1}(t)}$ for $t \in (0, \infty),$ and hence the critical point of the loss function $Q_{\sigma, \bA, \bw}$ is described by the equation
\be \label{eq:crit-point-3}
\sum_{j=1}^m w_j A_j^{-\fel} \ler{I-\ler{f_\sigma^{-1}\ler{A_j^{-\fel} X A_j^{-\fel}}}^{-1}}A_j^{-\fel}=0.
\ee

\subsection{The barycenter corresponding to the geometric mean}
For $\sigma=\#$ the generating function is $f_{\#}(x)=\sqrt{x},$ and hence the inverse is $f_{\#}^{-1}(t)=t^2.$ In this case, the critical point equation \eqref{eq:crit-point-3} describing the barycenter $\mathbf{bc}\ler{\#, \bA, \bw}$ reads as follows:
\be \label{eq:geom-bary}
\sum_{j=1}^m w_j \ler{A_j^{-1}-X^{-1} A_j X^{-1}}=0.
\ee
Note that \eqref{eq:geom-bary} may be considered as a \emph{generalized Riccati equation}, and in the special case $m=2, w_1=w_2=\fel,$ the solution of \eqref{eq:geom-bary} is the symmetric geometric mean $A_1 \# A_2.$
\par
More generally, if $m=2, w_1=1-\alpha,$ and $w_2=\alpha,$ then \eqref{eq:geom-bary} has the following form:
$$
(1-\alpha)A_1^{-1}+\alpha A_2^{-1}=(1-\alpha)X^{-1} A_1 X^{-1}+\alpha X^{-1} A_2 X^{-1},
$$
or equivalently
\be\label{eq:Riccati_uj-0}
X\left[(1-\alpha)A_1^{-1}+\alpha A_2^{-1}\right]X=(1-\alpha)A_1+\alpha A_2.
\ee
Recall that for positive definite $A$ and $B$, the Riccati equation 
$$
XA^{-1}X=B
$$
has a unique positive definite solution, that is the geometric mean 
$$
A\#B=A^{1/2}(A^{-1/2}BA^{-1/2})^{1/2}A^{1/2}.
$$
We can observe that \eqref{eq:Riccati_uj-0} is the Riccati equation for the weighted harmonic mean
$$
A_1 !_{\alpha} A_2=[(1-\alpha)A_1^{-1}+\alpha A_2^{-1}]^{-1}
$$
and the weighted arithmetic mean $A_1\nabla_{\alpha}A_2=(1-\alpha)A_1+\alpha A_2$, ie
\be\label{eq:Ric-0}
X(A_1 !_{\alpha} A_2)^{-1}X=A_1\nabla_{\alpha}A_2.
\ee
Hence the solution of \eqref{eq:Riccati_uj-0} is the geometric mean of the weighted harmonic and the weighted arithmetic mean
\be\label{eq:wgeom_uj-0}
X=(A_1 !_{\alpha} A_2)\#(A_1\nabla_{\alpha}A_2).
\ee
It means that in this case the weighted barycenter with respect to $\phi_{\#}$ does not coincide with the weighted geometric mean, nevertheless
$$
\mathbf{bc}\ler{\#, \lers{A_1, A_2}, \lers{1-\alpha, \alpha}}=(A_1 !_{\alpha} A_2)\#(A_1 \nabla_{\alpha} A_2)
$$
that is, $\mathbf{bc}\ler{\#, \lers{A_1, A_2}, \lers{1-\alpha, \alpha}}$ is the Kubo-Ando mean of $A_1$ and $A_2$ with representing function 
$$
f_{\ler{!_\alpha \# \nabla_\alpha}}(x)=
\sqrt{\frac{x(1-\alpha +\alpha x)}{(1-\alpha)x+\alpha}}.
$$
These means were widely investigated in \cite{kim-lawson-lim}.
\par
We note that the critical point equation \eqref{eq:geom-bary} can be rearranged as
\be \label{eq:geom-bary-v2}
X \ler{\sum_{j=1}^m w_j A_j^{-1}} X= \sum_{j=1}^m w_j A_j.
\ee
This is the Ricatti equation for the weighted multivariate harmonic mean $\ler{\sum_{j=1}^m w_j A_j^{-1}}^{-1}$ and arithmetic mean $\sum_{j=1}^m w_j A_j,$ hence the barycenter $\mathbf{bc}\ler{\#, \bA, \bw}$ coincides with the weighted $\cA \# \cH$-mean of Kim, Lawson, and Lim \cite{kim-lawson-lim}, that is,
\be \label{eq:bary-ah}
\mathbf{bc}\ler{\#, \bA, \bw}=\ler{\sum_{j=1}^m w_j A_j^{-1}}^{-1} \# \ler{\sum_{j=1}^m w_j A_j}.
\ee
\subsection{The weighted multivariate harmonic mean as a lower bound for the barycenter}
This subsection is devoted to show that the barycenter $\mathbf{bc}\ler{\sigma, \bA, \bw}$ is above the harmonic mean $\ler{\sum_{j=1}^m w_j A_j^{-1}}^{-1}$ in the L\"owner order.
\par
Let $f_\sigma$ be the representing function of the symmetric operator mean $\sigma$. It is well known that
\be\label{eq:f_ordering}
\frac{2x}{x+1}\le f_\sigma(x)\le \frac{1+x}{2},\qquad x \in (0, \infty).
\ee
The inverse of the function $\frac{2x}{x+1}$ is $\frac{x}{2-x}$ with the domain $(0,2)$. Since the graph of the inverse function can be obtained by 
reflecting the graph of the function across the line $y=x$, the reflection change the ordering, and we get
$$
f_{\sigma}^{-1}(x)\le\frac{x}{2-x},\quad x\in \ler{0,2},
$$
hence
\be\label{eq:f_inv_ordering}
\frac{2-x}{x}\le\frac{1}{f_{\sigma}^{-1}(x)},\quad x \in (0, \infty).
\ee
The critical point equation \eqref{eq:crit-point-3} for $Q_{\sigma, \bA, \bw}$ tells us that
\be \label{eq:crit-point-4}
\sum_{j=1}^m w_j A_j^{-1}=\sum_{j=1}^m w_j A_j^{-\fel}\ler{f_\sigma^{-1}\ler{A_j^{-\fel} \mathbf{bc}\ler{\sigma, \bA, \bw} A_j^{-\fel}}}^{-1}A_j^{-\fel}.
\ee
By \eqref{eq:f_inv_ordering}, the right hand side of \eqref{eq:crit-point-4} can be estimated from below by
$$
\sum_{j=1}^m w_j A_j^{-\fel} \ler{ 2 A_j^{\fel} \ler{\mathbf{bc}\ler{\sigma, \bA, \bw}}^{-1} A_j^{\fel}-I}A_j^{-\fel},
$$
hence
$$
\sum_{j=1}^m w_j A_j^{-1} \geq \sum_{j=1}^m w_j \ler{2 \ler{\mathbf{bc}\ler{\sigma, \bA, \bw}}^{-1} - A_j^{-1}}
$$
which is equivalent to the desired inequality
$$
\ler{\sum_{j=1}^m w_j A_j^{-1}}^{-1} \leq \mathbf{bc}\ler{\sigma, \bA, \bw}.
$$
\section{Further properties of the divergence $\phi_\sigma$ and the corresponding barycenter} \label{sec:further}

The following properties of $\phi_\sigma$ will be used in the sequel to show that $\phi_\sigma$ is symmetric if and only if $\sigma$ is the geometric mean (Proposition \ref{prop:symm}), and to show the congruence invariance of the barycenter $\mathbf{bc}\ler{\sigma, \bA, \bw}$ (Proposition \ref{prop:elem-prop}). These properties have been discussed in \cite[Sec. 2]{m-sz-laa-15} in detail, however, for the sake of completeness, we include the proofs, as well.

\begin{proposition} \label{prop:inv-cong}
For any Kubo-Ando mean $\sigma$ and for any $A,B\in\cB(\cH)^{++}$ we have
\be \label{eq:inv}
\phi_\sigma (A^{-1},B^{-1})=\phi_\sigma(B,A),
\ee
and the equality
\be \label{eq:congruence}
\phi_\sigma \left(TAT^*,TBT^*\right)=\phi_\sigma(A,B)
\ee
also holds for an arbitrary invertible operator $T\in\cB(\cH).$
\end{proposition}

\begin{proof}
Property \eqref{eq:inv} can easily be seen by the observation that $\mathrm{spec}\ler{A^{-\fel} B A^{-\fel}}=
\mathrm{spec}\ler{B^{\fel} A^{-1} B^{\fel}}.$
The proof of Property \eqref{eq:congruence} reads as follows:
\bec
\left[\left(TAT^*\right)^{-1/2}\left(TBT^*\right)\left(TAT^*\right)^{-1/2}\right]^2=\left(TAT^*\right)^{-1/2}TBA^{-1}BT^*\left(TAT^*\right)^{-1/2}=X^*X=\abs{X}^2,
\eec
where $X=A^{-1/2}BT^*\left(TAT^*\right)^{-1/2}$. Observe that we have
\be \label{eq:X}
\abs{X^*}^2=XX^*=A^{-1/2}BT^*\left(TAT^*\right)^{-1}TBA^{-1/2}=A^{-1/2}BA^{-1}BA^{-1/2}=\left(A^{-1/2}BA^{-1/2}\right)^2.
\ee
By using the polar decomposition $X=V\abs{X}$, $\abs{X^*}=V\abs{X}V^*$ and $\abs{X}=V^*\abs{X^*}V$ also hold. Hence we can write
\be
\ler{TAT^*}^{-1/2}\ler{TBT^*}\ler{TAT^*}^{-1/2}=\abs{X}=V^*\abs{X^*}V=V^*\ler{A^{-1/2}BA^{-1/2}}V,
\ee
where \eqref{eq:X} was used in the last step. Consequently we have
\bec
\phi_\sigma \ler{TAT^*,TBT^*}=\tr g_\sigma \left(V^*\ler{A^{-1/2}BA^{-1/2}}V\right)
\eec
\bec
=\tr V^*g_\sigma\ler{A^{-1/2}BA^{-1/2}}V=\tr g_\sigma\ler{A^{-1/2}BA^{-1/2}}=\phi_\sigma(A,B)
\eec
as we stated.
\end{proof}

\begin{proposition}\label{prop:symm}
The divergence $\phi_{\sigma}$ is symmetric in its arguments, that is 
$$\phi_{\sigma}(A,B)=\phi_{\sigma}(B,A)$$ holds for all $A,B\in\cB(\cH)^{++}$,  if and only if $\sigma=\#$ is the geometric mean.
\end{proposition}

\begin{proof}
If $\phi_{\sigma}(A,B)=\phi_{\sigma}(B,A)$ holds for all $A,B\in\cB(\cH)^{++}$, then by the Theorem \ref{thm:bary} we have
\begin{eqnarray*}
A\sigma B&=&\argmin_{X\in\cB(\cH)^{++}}\frac{1}{2}\left(\phi_\sigma (A,X)+\phi_\sigma(B,X)\right)\\
&=&\argmin_{X\in\cB(\cH)^{++}}\frac{1}{2}\left(\phi_\sigma (X,A)+\phi_\sigma(X,B)\right)\\
&=&\argmin_{X\in\cB(\cH)^{++}}\frac{1}{2}\left(\phi_\sigma \ler{A^{-1},X^{-1}}+\phi_\sigma\ler{B^{-1},X^{-1}}\right)=\ler{A^{-1}\sigma B^{-1}}^{-1},
\end{eqnarray*}
where \eqref{eq:inv} was used in the second line. Recall that the adjoint $\sigma^*$ of the mean $\sigma$ is defined for invertible $A, B$ by $A\sigma^*B=\ler{A^{-1}\sigma B^{-1}}^{-1}$, and its representing function is given by 
$$
f_{\sigma^*}(x)=\frac{1}{f_{\sigma}\left(\frac{1}{x}\right)}.
$$
Hence it is necessary that the mean $\sigma$ is selfadjoint, that is $A\sigma B=A\sigma^*B$, which means that for the representation function we have 
$$
f_{\sigma}(x)=\frac{1}{f_{\sigma}\left(\frac{1}{x}\right)}.
$$ Since $\sigma$ is also symmetric, i.e. $f_{\sigma}(x)=xf_{\sigma}(1/x)$, these two conditions implies that $f_{\sigma}(x)=\sqrt{x}$, hence we have geometric mean. The sufficiency is clear as
$$
\phi_\#(A,B)= \tr\ler{A^{-\fel}BA^{-\fel}+A^{\fel}B^{-1}A^{\fel}-2I}.
$$
\end{proof}

\subsection{Elementary properties of the barycenter}
We turn to show some of the desirable properties of the barycenter $\mathbf{bc}\ler{\sigma, \bA, \bw}.$

\begin{proposition} \label{prop:elem-prop}
The barycenter $\mathbf{bc}\ler{\sigma, \bA, \bw}$ defined in \eqref{eq:bary_alpha} satisfies the following properties:
\ben[label=(\alph*)]
\item \label{prop:idemp}
Idempotency: $\mathbf{bc}\ler{\sigma, \lers{A,\dots,A}, \bw}=A$ for any symmetric Kubo-Ando mean $\sigma,$ any $A \in \bh^{++},$ and any probability vector $\bw.$
\item \label{prop:homog}
Homogeneity: $\mathbf{bc}\ler{\sigma, t \bA, \bw}=t \mathbf{bc}\ler{\sigma, \bA, \bw}$ where the shorthand $t \bA$ denotes $\lers{t A_1, \dots, t A_m}$ if $\bA=\lers{A_1, \dots \, A_m}$
\item \label{prop:perm}
Permutation invariance: $\mathbf{bc}\ler{\sigma, \bA_\pi, \bw_\pi}=\mathbf{bc}\ler{\sigma, \bA, \bw}$ where $\pi$ is a permutation of $\lers{1, \dots, m},$ and $\bA_\pi=\lers{A_{\pi(1)}, \dots, A_{\pi(m)}}, \, \bw_\pi=\lers{w_{\pi(1)}, \dots, w_{\pi(m)}}.$
\item \label{prop:cong}
Congruence invariance:
$$
\mathbf{bc}\ler{\sigma, T \bA T^{*}, \bw}=T \mathbf{bc}\ler{\sigma, \bA, \bw} T^{*}
$$
for any invertible $T \in \bh,$ where $T \bA T^{*}=\lers{T A_1, T^{*}, \dots, T A_m T^{*}}$ if $\bA=\lers{A_1, \dots \, A_m}.$
\een
\end{proposition}

\begin{proof}
Properties \ref{prop:idemp},\ref{prop:homog}, \ref{prop:perm} are straightforward. To justify \ref{prop:cong}, we use the second statement of Proposition \ref{prop:inv-cong}, that is, eq. \eqref{eq:congruence}. By this congruence invariant property of $\phi_\sigma$, it is clear that the unique minimizer of the loss function $Q_{\sigma, T \bA T^{*}, \bw}$ is $T X T^{*},$ where $X$ is the unique minimizer of $Q_{\sigma, \bA, \bw}.$
\end{proof}


\subsection{In-betweenness}
We finish with an in-betweenness result for the divergences $\phi_\sigma$ and for arbitrary Kubo-Ando means.
The concept of in-betweenness was introduced by Audenaert \cite{aud-laa-13} and has been investigated later by Dinh, Dumitru, and Franco \cite{d-d-f-inb-19, du-fra-laa-20} among others.

\begin{proposition}
For any Kubo-Ando means $\sigma$ and $\tau,$ the in-betweenness property
$$
\phi_\sigma(A, A \tau B) \leq \phi_\sigma (A, B) \qquad \ler{A, B \in \mathcal{B}(\mathcal{H})^{++}}
$$
holds, where $\phi_\sigma$ is defined by \eqref{eq:fi-sig}.
\end{proposition}

\begin{proof}

For any positive number $x,$ the value $f_\tau(x)$ --- where $f_\tau$ is the operator monotone generating function of the mean $\tau$ --- is between $x$ and $1,$ because $f_\tau(x)$ is an integral of weighted harmonic means of $x$ and $1.$ That is, for $x \leq 1$ we have $x\leq f_\tau(x)\leq 1,$ and for $1\leq x$ we have $1\leq f_\tau(x) \leq x.$ As noted before, $g_\sigma$ is positive, convex, and has a unique zero at $1$. Consequently, $g_\sigma\left(f_\tau(x)\right) \leq g_\sigma(x)$ for all $x \in (0, \infty),$ and hence
$$
g_\sigma\left(f_\tau\left(A^{-\frac{1}{2}}BA^{-\frac{1}{2}}\right)\right)
\leq
g_\sigma\left(A^{-\frac{1}{2}}BA^{-\frac{1}{2}}\right) \qquad \ler{A, B \in \mathcal{B}(\mathcal{H})^{++}},
$$
so
$$
\phi_\sigma (A, A \tau B)=\tr g_\sigma \left(f_\tau\left(A^{-\frac{1}{2}}BA^{-\frac{1}{2}}\right)\right)
\leq
\tr g_\sigma\left(A^{-\frac{1}{2}}BA^{-\frac{1}{2}}\right)=\phi_\sigma(A, B).
$$
\end{proof}

\paragraph*{{\bf Acknowledgement}} The authors are grateful to Mil\'an Mosonyi for fruitful discussions on the topic, and to the anonymous referee for his/her comments and suggestions.

\section{Appendix}\label{sec: Appendix}

One can define a Riemannian metric on $\cB(\cH)^{++}$ locally at $A$ by the relation 
$$\ds=\|A^{-1}\dd A A^{-1}\|_2,$$
where $\|X\|=\sqrt{\tr X^*X}$ is the Hilbert-Schmidt or Frobenius norm of $X$. Let $X: [a,b]\to \cB(\cH)^{++} $ be a smooth path. The arc-length along this path is given by
\be\label{eq: L}
L(X)=\int_a^b\|X(t)^{-1}X'(t)X(t)^{-1}\|_2\dt
\ee
The corresponding geodesic distance between $A,B\in \cB(\cH)^{++}$ is defined by 
\be\label{eq:geod}
\delta(A,B)=\inf\left\{ \int_0^1\|X(t)^{-1}X'(t)X(t)^{-1}\|_2\dt :X(t)\in \cB(\cH)^{++} \text{ for } t\in (0,1), X(0)=A, X(1)=B\right\}.
\ee
Given any smooth path $X(t)$ joining $A$ to $B$ define $H(t)=X(t)^{-1}$. Then $$H'(t)=-X(t)^{-1}X'(t)X(t)^{-1}$$ and 
\begin{eqnarray*}
&& \norm{B^{-1}-A^{-1}}_2=\norm{H(1)-H(0)}_2=\norm{\int_0^1H'(t)\dd t}_2\\
 &=&\norm{-\int_0^1 X(t)^{-1}X'(t)X(t)^{-1}\dd t}_2
 \le \int_0^1\norm{X(t)^{-1}X'(t)X(t)^{-1}}_2\dd t=L(X).
\end{eqnarray*}
We show that the weighted harmonic mean
$$
\gamma (t)=A!_tB=\left[(1-t)A^{-1}+tB^{-1}\right]^{-1}, \quad t\in [0,1]
$$ 
is the unique constant speed geodesic running from $A$ to $B$ in unit time. 

Denote $C=A^{1/2}B^{-1}A^{1/2}$ and $X(t)=(1-t)I+tC$. Then we can write
$$\gamma(t)=[A^{-1/2}((1-t)I+tA^{1/2}B^{-1}A^{1/2})A^{-1/2}]^{-1}=A^{1/2}[(1-t)I+tC]^{-1}A^{1/2}=A^{1/2}X(t)^{-1}A^{1/2},$$
$$\gamma'(t)=-A^{1/2}[(1-t)I+tC]^{-1}(C-I)[(1-t)I+tC]^{-1}A^{1/2}=-A^{1/2}X(t)^{-1}(C-I)X(t)^{-1}A^{1/2},$$
$$\gamma(t)^{-1}=A^{-1/2}X(t)A^{-1/2}.$$
After substitution we get
$$\|\gamma(t)^{-1}\gamma'(t)\gamma(t)^{-1}\|_2^2=\tr\left\{ \gamma(t)^{-1}\gamma'(t)\gamma(t)^{-2}\gamma'(t)\gamma(t)^{-1}\right\}$$
$$=\tr\left\{ A^{-1/2}X(t)A^{-1/2} A^{1/2}X(t)^{-1}(C-I)X(t)^{-1}A^{1/2}A^{-1/2}X(t)A^{-1}X(t)A^{-1/2}\times\right. $$ 
$$\left.A^{1/2}X(t)^{-1}(C-I)X(t)^{-1}A^{1/2} A^{-1/2}X(t)A^{-1/2}\right\}$$
$$=\tr\left\{ A^{-1/2}(C-I)A^{-1}(C-I)A^{-1/2}\right\} = \tr (B^{-1}-A^{-1})(B^{-1}-A^{-1})=\|B^{-1}-A^{-1}\|_2^2.
$$
It shows that the weighted harmonic mean $\gamma (t)=A!_tB=[(1-t)A^{-1}+tB^{-1}]^{-1}$ is geodesic and its length is
$$L(\gamma)=\int_0^1\|\gamma(t)^{-1}\gamma'(t)\gamma(t)^{-1}\|_2\dd t=\int_0^1\|B^{-1}-A^{-1}\|_2\dd t=\|B^{-1}-A^{-1}\|_2,$$
which also means that the geodesic distance is $\delta (A,B)=\|B^{-1}-A^{-1}\|_2$.


\begin{thebibliography}{9}

\bibitem{amari-book} S. Amari, \emph{Information Geometry and its Applications,} Springer (Tokyo), 2016.

\bibitem{ando-laa-79} T. Ando, \emph{Concavity of certain maps on positive definite matrices and applications to Hadamard products,} Linear Algebra Appl. {\bf 26} (1979), 203-241.

\bibitem{kubo-ando} T. Ando, F. Kubo, \emph{Means of positive linear operators}, Math. Ann. {\bf 246} (1980), 205–224.

\bibitem{ando-hiai} T. Ando, F. Hiai, \emph{Operator log-convex functions and operator means,} Math. Ann. {\bf 350} (2011), 611-630.

\bibitem{aud-laa-13} K.M.R. Audenaert, \emph{In-betweenness, a geometrical monotonicity property for operator means,} Linear Algebra Appl. {\bf 438} (2013), 1769-1778.

\bibitem{bhatia} R. Bhatia, \emph{Matrix Analysis,} Springer-Verlag, New York, 1997.


\bibitem{bhat-holb} R. Bhatia, J.A.R. Holbrook, \emph{Riemannian geometry and matrix geometric means,} Linear Algebra Appl. {\bf 413} (2006), 594–618.

\bibitem{carlen} E. Carlen, Trace inequalities and quantum entropy: an introductory course, \emph{Contemp. Math.} {\bf 529} (2010), 73--140.

\bibitem{d-d-f-inb-19} T. H. Dinh, R. Dumitru, J. A. Franco, \emph{Some geometric properties of matrix means with respect to different distance function,} arXiv preprint, arXiv:1910.04833

\bibitem{du-fra-laa-20} R. Dumitru, J. Franco, \emph{Generalized Hellinger metric and Audenaert's in-betweenness,} Linear Algebra Appl. {\bf 585} (2020), 191--198.

\bibitem{hansen-conv-mom} F. Hansen, \emph{Convex multivariate operator means,} Linear Algebra Appl. {\bf 564} (2019), 209--224.

\bibitem{HPbook} F. Hiai, D. Petz, \emph{Introduction to Matrix Analysis and Applications}, Hindustan Book Agency and Springer Verlag, 2014.

\bibitem{kim-lawson-lim} S. Kim, J. Lawson, Y. Lim, \emph{The matrix geometric mean of parametrized, weighted arithmetic and harmonic means,} Linear Algebra Appl. {\bf 435} (2011), 2114--2131.

\bibitem{lim-palfia-jfa} Y. Lim, M. P\'alfia, {\it Matrix power means and the Karcher mean,} J. Funct. Anal. {\bf 262} (2012), 1498-1514.

\bibitem{m-sz-laa-15} L. Moln\'ar, P. Szokol, \emph{Transformations on positive definite matrices preserving generalized distance measures,} Linear Algebra Appl. {\bf 466} (2015), 141--159.

\bibitem{pv-15} J. Pitrik, D. Virosztek, {\it On the joint convexity of the Bregman divergence of matrices,} Lett. Math. Phys. {\bf 105} (2015), 675--692.

\bibitem{pusz-woron} W. Pusz, S. L. Woronowicz, \emph{Functional calculus for sesquilinear forms and the purification map,} Rep. Math. Phys. {\bf 8} (1975), 159-170.

\bibitem{sra-pams} S. Sra, \emph{Positive definite matrices and the $S$-divergence,} Proc. Amer. Math. Soc. {\bf 144} (2016), 2787-2797.

\end{thebibliography}
\end{document}